\documentclass[preprint,12pt]{elsarticle}

\usepackage{amssymb}
 \usepackage{graphicx}
\usepackage{epsfig}

\usepackage{amsthm}

\newcommand{\sysn}{\left\{\begin{array}{rcl}}
\newcommand{\sysk}{\end{array}\right.}

\newtheorem{theorem}{Theorem}[section]

\theoremstyle{example}
\newtheorem{example}[theorem]{Example}

\newtheorem{proposition}[theorem]{Proposition}
\theoremstyle{definition}
\newtheorem{definition}[theorem]{Definition}
\newtheorem{remark}[theorem]{Remark}
\newtheorem{corollary}[theorem]{Corollary}

\journal{...}

\begin{document}

\title{Some observations on the mildly Menger property and topological games}

\author[affil1]{Manoj Bhardwaj}

\address[affil1]{Department of Mathematics, University of Delhi, New Delhi-110007, India}

\ead[affil1]{manojmnj27@gmail.com}

\author[affil2]{Alexander V. Osipov}

\address[affil2]{Krasovskii Institute of Mathematics and Mechanics, \\ Ural Federal
 University, Ural State University of Economics, Yekaterinburg, Russia}

\ead[affil2]{OAB@list.ru}

\begin{abstract}

In this paper, we defined two new games - the mildly Menger game
and the compact-clopen game. In a zero-dimensional space, the
Menger game is equivalent to the mildly Menger game and the
compact-open game is equivalent to the compact-clopen game. An
example is given for a space on which the mildly Menger game is
undetermined. Also we introduced a new game namely
$\mathcal{K}$-quasi-component-clopen game and proved that this
game is equivalent to the compact-clopen game. Then we proved that
if a topological space is an union of countably many
quasi-components of compact sets, then TWO has a winning strategy
in the mildly Menger game.

\end{abstract}


\maketitle 

\section{Introduction}\label{sec1}

 In 1924, Menger \cite{H47} (see also \cite{H4}) introduced
covering property in topological spaces. A space $X$ is said to
have \textit{Menger property} if for each sequence $\langle
\mathcal{U}_n : n \in \omega \rangle$ of open covers of $X$ there
is a sequence $\langle \mathcal{V}_n : n \in \omega \rangle$ such
that for each $n$, $\mathcal{V}_n$ is a finite subset of
$\mathcal{U}_n$ and each $x \in X$ belongs to $\bigcup
\mathcal{V}_n$ for some $n$.

In covering properties, Menger property is one of the most
important property. This property is stronger than Lindel\"{o}f
and weaker than $\sigma$- compactness.

Usually, each selection principle $S_{fin}(\mathcal{A},
\mathcal{B})$ can be associated with some topological game
$G_{fin}(\mathcal{A},\mathcal{B})$. So the Menger property
$S_{fin}(\mathcal{O}, \mathcal{O})$  is associated with the Menger
game $G_{fin}(\mathcal{O}, \mathcal{O})$.

In \cite{H4} Hurewicz proved that a topological space $X$ is
Menger if and only if ONE does not have a winning strategy in the
Menger game on $X$. Thus, the Menger property can be investigated
from the point of view of topological game theory.

In (\cite{Tel}, Corollary 3), R. Telg\'{a}rsky proved that ONE has
a winning strategy in the compact-open game if and only if TWO has
a winning strategy in the Menger game. Telg\'{a}rsky also observes
(Proposition 1, \cite{Tel}) ONE having a winning strategy in the
Menger game implies TWO having a winning strategy in the
compact-open game.

Lj. D. Ko\v{c}inac define and study a version of the classical
Hurewicz covering property by using clopen covers. He call this
property {\it mildly Hurewicz}. In \cite{Ko}, game-theoretic and
Ramsey-theoretic characteristics of this property are given.

In this paper, we define two new games - the mildly Menger game
and the compact-clopen game. In a zero-dimensional space, the
Menger game is equivalent to the mildly Menger game and the
compact-open game is equivalent to the compact-clopen game.  Also
we introduced a new game namely
$\mathcal{K}$-quasi-component-clopen game and proved that this
game is equivalent to the compact-clopen game.

\section{Preliminaries}\label{sec2}

Let $(X,\tau)$ or $X$ be a topological space. We will denote by
$Cl(A)$ and $Int(A)$ the closure of $A$ and the interior of $A$,
for a subset $A$ of $X$, respectively. If a set is open and closed
in a topological space, then it is called {\it clopen}.  Recall
that a space $X$ is called {\it zero-dimensional} if it is
nonempty and has a base consisting of clopen sets, i.e., if for
every point $x\in X$ and for every neighborhood $U$ of $x$ there
exists a clopen subset $C\subseteq X$ such that $x\in C\subseteq
U$. It is clear that a nonempty subspace of a zero-dimensional
space is again zero-dimensional.

Note that separable zero-dimensional metric spaces are
homeomorphic to subsets of the irrational numbers (\cite{enc},[E,
6.2.16]). For the terms and symbols that we do not define follow
\cite{H10}.

Let $\mathcal{A}$ and $\mathcal{B}$ be collections of open covers of a topological space $X$.

The symbol $S_{fin}(\mathcal{A}, \mathcal{B})$ denotes the
selection principle that for each sequence $\langle \mathcal{U}_n
: n \in \omega \rangle$ of elements of $\mathcal{A}$ there exists
a sequence $\langle \mathcal{V}_n : n \in \omega \rangle$ such
that for each $n$, $\mathcal{V}_n$ is a finite subset of
$\mathcal{U}_n$ and $\bigcup_{n \in \omega} \mathcal{V}_n$ is an
element of $\mathcal{B}$ \cite{H1}.

In this paper $\mathcal{A}$ and $\mathcal{B}$ will be collections of the following open covers of a space $X$:

$\mathcal{O}$ : the collection of all open covers of $X$.

$\mathcal{C}_\mathcal{O}$ : the collection of all clopen covers of $X$.

\medskip

Clearly, $X$ has the Menger property if and only if $X$ satisfies
$S_{fin}(\mathcal{O}, \mathcal{O})$.

\begin{definition} \label{2.2}
A space $X$ is said to have \textit{mildly Menger property} if for
each  sequence $\langle \mathcal{U}_n : n \in \omega \rangle$ of
clopen covers of $X$ there is a sequence $\langle \mathcal{V}_n :
n \in \omega \rangle$ such that for each $n$, $\mathcal{V}_n$ is a
finite subset of $\mathcal{U}_n$ and each $x \in X$ belongs to
$\bigcup \mathcal{V}_n$ for some $n$, i.e., $X$ satisfies
$S_{fin}(\mathcal{C}_\mathcal{O}, \mathcal{C}_\mathcal{O})$.
\end{definition}

The proof of the following result easily follows from replacing
the open sets with sets of a clopen base of the topological space.

\begin{theorem} \label{a1}
For a zero-dimensional space $X$, $S_{fin}(\mathcal{C}_\mathcal{O}, \mathcal{C}_\mathcal{O})$ is equivalent to $S_{fin}(\mathcal{O}, \mathcal{O})$.
\end{theorem}

\section{Games related to $S_{fin}(\mathcal{O}, \mathcal{O})$ and $S_{fin}(\mathcal{C}_\mathcal{O}, \mathcal{C}_\mathcal{O})$}

 The {\it selection game} $G_{fin}(\mathcal{A},\mathcal{B})$ is an
 $\omega$-length game played by two players, ONE and TWO. During round $n$,
 ONE choose $A_n\in \mathcal{A}$, followed by TWO choosing
 $B_n\in [A_n]^{<\omega}$. Player TWO wins in the case that $\bigcup\{B_n:
 n<\omega\}\in \mathcal{B}$, and Player ONE wins otherwise.

We consider the following selection games:

\medskip

$\bullet$  $G_{fin}(\mathcal{O}, \mathcal{O})$ - the {\it Menger
game}.

\medskip

$\bullet$ $G_{fin}(\mathcal{C}_\mathcal{O},
\mathcal{C}_\mathcal{O})$ - the {\it mildly Menger game}.

\medskip

In \cite{H4} Hurewicz proves:

\begin{theorem}\label{hu}(Hurewicz) A topological space has the Menger property $S_{fin}(\mathcal{O},
\mathcal{O})$ if, and only if, ONE has no winning strategy in the
Menger game $G_{fin}(\mathcal{O}, \mathcal{O})$.
\end{theorem}

Telg\'{a}rsky proved that a metric space $X$ is $\sigma$-compact
if, and only if, TWO has a winning strategy in the Menger game.

If player has a winning strategy, we write $Player \uparrow
G_{fin}(\mathcal{A},\mathcal{B})$. If player has no winning
strategy, we write $Player \not\uparrow
G_{fin}(\mathcal{A},\mathcal{B})$.

Note that the following chain of implications always holds:

\begin{center}

$X$ is $\sigma$-compact \\
$\Downarrow$ \\ $TWO \uparrow G_{fin}(\mathcal{O}, \mathcal{O})$ $\Leftrightarrow$ $ONE \uparrow$ the compact-open game \\ $\Downarrow$ \\ $ONE \not\uparrow G_{fin}(\mathcal{O}, \mathcal{O})$ \\ $\Updownarrow$ \\
$X$ has the Menger property.

\end{center}


The {\it compact-open game} ({\it compact-clopen game}) on a space
$X$ is played according to the following rules:

In each inning $n \in \omega$, ONE picks a compact set $K_n
\subseteq X$, and then TWO chooses an open (clopen) set $U_n
\subseteq X$ with $K_n \subseteq U_n$. At the end of the play
\begin{center}
$K_0, U_0, K_1, U_1, K_2, U_2, . . . , K_n, U_n, . . . $,
\end{center}
the winner is ONE if $X \subseteq \bigcup_{n \in \omega} U_n$, and
TWO otherwise.

\bigskip

Let $\mathcal{K}$ denotes the collection of all compact subsets of
a space $X$. We denote the collection of all clopen subsets of a
space by $\tau_{c}$ and the collection of all finite subsets of
$\tau_{c}$ by $\tau^{< \omega}_{c}$.

A strategy for ONE in the compact-clopen game on a space $X$ is a
function $\varphi : \tau^{< \omega}_{c} \rightarrow \mathcal{K}$.

A strategy for TWO in the compact-clopen game on a space $X$ is a
function $\psi : \mathcal{K}^{< \omega} \rightarrow \tau_{c}$ such
that, for all $\langle K_0, K_1,..., K_n \rangle \in
\mathcal{K}^{< \omega} \setminus \{\langle \rangle\}$, we have
$K_n \subseteq \psi(\langle K_0,..., K_n \rangle$) = $U_n$.

A strategy $\varphi : \tau^{< \omega}_{c} \rightarrow \mathcal{K}$
for ONE in the compact-clopen game on $X$ is a winning strategy
for ONE if, for every sequence $\langle U_n : n \in \omega
\rangle$ of clopen subsets of a space $X$ such that $\forall n \in
\omega$, $K_n = \varphi(\langle U_0, U_1,..., U_{n-1} \rangle)
\subseteq U_n$, we have $X \subseteq \bigcup_{n \in \omega} U_n$.

A strategy $\psi : \mathcal{K}^{< \omega} \rightarrow \tau_{c}$
for TWO in the compact-clopen game on $X$ is a winning strategy
for TWO if, for every sequence $\langle K_n : n \in \omega
\rangle$ of compact subsets of a space $X$, we have $X \subseteq
\bigcup_{n \in \omega} (\psi(\langle K_0, K_1,..., K_n \rangle) =
U_n)$.

\medskip
Recall that two games $G$ and $G^{'}$ are equivalent (isomorphic)
if
\begin{enumerate}
\item ONE has a winning strategy in $G$ if and only if ONE has a
winning strategy in $G^{'}$; \item TWO has a winning strategy in
$G$ if and only if TWO has a winning strategy in $G^{'}$.
\end{enumerate}

\medskip

The proof of the following result easily follows from replacing
the open sets with sets of a clopen base of the topological space.

\begin{theorem} \label{a4}
For a zero-dimensional space, the following statements hold:
\begin{enumerate}
\item
The game $G_{fin}(\mathcal{C}_\mathcal{O}, \mathcal{C}_\mathcal{O})$ is equivalent to the game $G_{fin}(\mathcal{O}, \mathcal{O})$.
\item
The compact-clopen game is equivalent to the compact-open game.
\end{enumerate}
\end{theorem}

Recall that a topological space $X$ is  {\it mildly compact}, if
every clopen cover of $X$ contains a finite subcover; and {\it
mildly Lindel\"{o}f} if every clopen cover has a countable
subcover \cite{st}. A space $X$ is a {\it $\sigma$-mildly compact
space}, if $X=\bigcup_{i\in \omega} A_i$ where $A_i$ is a mildly
compact space for all $i\in \omega$.

Note that the mildly Menger property is stronger than mildly
Lindel\"{o}f and weaker than $\sigma$--mildly compactness.

The {\it mildly compact-clopen game} on a space $X$ is played
according to the following rules :

In each inning $n \in \omega$, ONE picks a mildly compact set $K_n
\subseteq X$, and then TWO chooses a clopen set $U_n \subseteq X$
with $K_n \subseteq U_n$. At the end of the play
\begin{center}
$K_0, U_0, K_1, U_1, K_2, U_2, . . . , K_n, U_n, . . . $,
\end{center}
the winner is ONE if $X \subseteq \bigcup_{n \in \omega} U_n$, and
TWO otherwise.

\begin{theorem} \label{d1}
For a topological space $X$ the following statements hold:
\begin{enumerate}
\item If ONE has a winning strategy in the compact-clopen game,
then ONE has a winning strategy in the mildly compact-clopen game
on $X$. \item If ONE has a winning strategy in the mildly
compact-clopen game, then TWO has a winning strategy in the game
$G_{fin}(\mathcal{C}_\mathcal{O}, \mathcal{C}_\mathcal{O})$ on
$X$. \item If $X$ is a subset of the irrational numbers (a
zero-dimensional second-countable space) and TWO has a winning
strategy in the game $G_{fin}(\mathcal{C}_\mathcal{O},
\mathcal{C}_\mathcal{O})$, then $X$ is $\sigma$-compact. \item If
$X$ is $\sigma$-compact, then ONE has a winning strategy in the
compact-clopen game. \item If ONE has a winning strategy in the
game $G_{fin}(\mathcal{C}_\mathcal{O}, \mathcal{C}_\mathcal{O})$,
then TWO has a winning strategy in the mildly compact-clopen game
on $X$. \item If TWO has a winning strategy in the mildly
compact-clopen game, then TWO has a winning strategy in the
compact-clopen game. \item If $X$ is a zero-dimensional space and
TWO has a winning strategy in the compact-clopen game, then TWO
has a winning strategy in the mildly compact-clopen game.
\end{enumerate}
\end{theorem}

The following diagrams could be helpful in order to show the big
picture where $C.$ $CL(X)$ and $MC.$ $CL(X)$ are designations for
the compact-clopen game and the mildly compact-clopen game,
respectively.

\begin{center}
$X$ is $\sigma$-compact  \, \, \, \, \, \, \, \, \, \, \, \, \, \, \, \, \, \, \, \, \, \, \, \, \, \, \, \, \, \, \, \, \, \, \, \, \, \, \, \, \, \, \, \, \\
$\Downarrow$ (4) \, \, \, \, \, \, \, \, \, \, \, \, \, \, \, \,
\, \, \, \, \, \, \,
\, \, \, \, \,  \, \, \, \, \, \, \, \, \, \, \, \, \, \, \, \, \\
$ONE \uparrow C.$ $CL(X)$ ${\Rightarrow\atop (1)}$ $ONE \uparrow
MC.$ $CL(X)$ ${\Rightarrow\atop (2)}$
$TWO \uparrow G_{fin}(\mathcal{C}_\mathcal{O}, \mathcal{C}_\mathcal{O})$ \\
$\Downarrow$ \, \, \, \, \, \, \, \, \, \, \, \, \, \, \, \, $\Downarrow$ \, \, \, \, \, \, \, \, \, \, \, \, \, \, \, \, $\Downarrow$ \\
$TWO \not\uparrow C.$ $CL(X)$ ${\Rightarrow\atop (6)}$ $TWO
\not\uparrow MC.$ $CL(X)$ ${\Rightarrow\atop (5)}$ $ONE
\not\uparrow$ $G_{fin}(\mathcal{C}_\mathcal{O},
\mathcal{C}_\mathcal{O})$.
\end{center}

\medskip

 $TWO \uparrow G_{fin}(\mathcal{C}_\mathcal{O},
\mathcal{C}_\mathcal{O})$ + $(X\subseteq \omega^{\omega})$
${\Rightarrow\atop (3)}$ $X$ is $\sigma$-compact.

\medskip

\medskip

 $TWO \not\uparrow MC.$ $CL(X)$+$(X$ is a zero-dim. space$)$
${\Rightarrow\atop (7)}$ $TWO \not\uparrow C.$ $CL(X)$.

\begin{proof}
\begin{enumerate}
\item The proof follows from the fact that every compact subset is
mildly compact. \item Consider a winning strategy $\varphi$ for
ONE in the mildly compact-clopen game. To obtain a winning
strategy, we use $\varphi$ for TWO in the game
$G_{fin}(\mathcal{C}_\mathcal{O}, \mathcal{C}_\mathcal{O})$ on
$X$.

ONE starts $G_{fin}(\mathcal{C}_\mathcal{O},
\mathcal{C}_\mathcal{O})$ with his initial move $\mathcal{U}_0$, a
cover by clopen sets of $X$. Then TWO replies with a finite subset
$\mathcal{V}_0$ of $\mathcal{U}_0$ such that $K_0 =
\varphi(\langle \rangle) \subseteq \bigcup \mathcal{V}_0$.

If ONE plays $\mathcal{U}_n$ in $n$th inning, then TWO replies
with a finite subset $\mathcal{V}_n$ of $\mathcal{U}_n$ such that
$K_n = \varphi(\langle \mathcal{V}_0, \mathcal{V}_1,
\mathcal{V}_2,...,\mathcal{V}_{n-1} \rangle) \subseteq \bigcup
\mathcal{V}_n$.

In the same manner, the sets $ \mathcal{V}_0, \mathcal{V}_1,
\mathcal{V}_2,...,\mathcal{V}_n,... $ played by TWO in the play of
the game $G_{fin}(\mathcal{C}_\mathcal{O},
\mathcal{C}_\mathcal{O})$ are same as played by TWO in the
following play of the mildly compact-clopen game on $X$ :
\begin{center}
$\langle K_0 = \varphi(\langle \rangle), \mathcal{V}_0, K_1 = \varphi(\langle \mathcal{V}_0 \rangle), \mathcal{V}_1,..., K_n = \varphi(\langle \mathcal{V}_0, \mathcal{V}_1, \mathcal{V}_2,...,\mathcal{V}_{n-1} \rangle), \mathcal{V}_n,... \rangle$.
\end{center}
In the above play of the mildly compact-clopen game on $X$, ONE
uses his winning strategy $\varphi$, so $\bigcup_{n \in \omega}
\bigcup \mathcal{V}_n = X$. This implies that
\begin{center}
$\langle \mathcal{U}_0, \mathcal{V}_0, \mathcal{U}_1, \mathcal{V}_1,..., \mathcal{U}_n, \mathcal{V}_n,... \rangle$
\end{center}
is a play of the $G_{fin}(\mathcal{C}_\mathcal{O},
\mathcal{C}_\mathcal{O})$ on $X$ in which TWO has a winning
strategy.

\item If TWO has a winning strategy in
$G_{fin}(\mathcal{C}_\mathcal{O}, \mathcal{C}_\mathcal{O})$, then
TWO has a winning strategy in $G_{fin}(\mathcal{O}, \mathcal{O})$
by Theorem \ref{a4}. Rest of the proof follows from Theorem 1 \cite{H28b}

\item The proof is obvious. \item
 Consider a winning strategy $\varphi$ for ONE in $G_{fin}(\mathcal{C}_\mathcal{O}, \mathcal{C}_\mathcal{O})$. To obtain a winning strategy, we
 use $\varphi$ for TWO in the mildly compact-clopen game on $X$.

ONE starts the mildly compact-clopen game with his initial move
$K_0$, a mildly compact subset of $X$. Then TWO replies with
$\bigcup \mathcal{V}_0$ containing $K_0$ such that $\mathcal{V}_0$
is a finite subset of $\mathcal{U}_0 = \varphi(\langle \rangle)$.

If ONE plays $K_1$ his next move, then TWO replies with $\bigcup
\mathcal{V}_1$ containing $K_1$ such that $\mathcal{V}_1$ is a
finite subset of $\mathcal{U}_1  = \varphi(\langle \mathcal{V}_0
\rangle)$.

If ONE plays $K_2$ his next move, then TWO replies with $\bigcup
\mathcal{V}_2$ containing $K_2$ such that $\mathcal{V}_2$ is a
finite subset of $\mathcal{U}_2 = \varphi(\langle \mathcal{V}_0,
\mathcal{V}_1 \rangle)$ and so on.

If ONE plays $K_n$ in $n$th inning, then TWO replies with $\bigcup
\mathcal{V}_n$ containing $K_n$ such that $\mathcal{V}_n$ is a
finite subset of $\mathcal{U}_n = \varphi(\langle \mathcal{V}_0,
\mathcal{V}_1, \mathcal{V}_2,...,\mathcal{V}_{n-1} \rangle)$.

In the same manner, the sets $ \bigcup \mathcal{V}_0, \bigcup
\mathcal{V}_1, \bigcup \mathcal{V}_2,..., \bigcup
\mathcal{V}_n,... $ played by TWO in the play of mildly
compact-clopen game are same as played by TWO in the following
play of the $G_{fin}(\mathcal{C}_\mathcal{O},
\mathcal{C}_\mathcal{O})$ on $X$ :
\begin{center}
$\langle \mathcal{U}_0 = \varphi(\langle \rangle), \mathcal{V}_0, \mathcal{U}_1 = \varphi(\langle \mathcal{V}_0 \rangle), \mathcal{V}_1,..., \mathcal{U}_n = \varphi(\langle \mathcal{V}_0, \mathcal{V}_1, \mathcal{V}_2,...,\mathcal{V}_{n-1} \rangle), \mathcal{V}_n,... \rangle$.
\end{center}
In the above play of $G_{fin}(\mathcal{C}_\mathcal{O},
\mathcal{C}_\mathcal{O})$ on $X$, ONE uses his winning strategy
$\varphi$, so $\bigcup_{n \in \omega} \bigcup \mathcal{V}_n \neq
X$. This implies that
\begin{center}
$\langle K_0, \bigcup \mathcal{V}_0, K_1, \bigcup \mathcal{V}_1,..., K_n, \bigcup \mathcal{V}_n,... \rangle$
\end{center}
is a play of the mildly compact-clopen game on $X$ in which TWO
has a winning strategy. \item The proof is obvious. \item The
proof follows from the fact that in a zero-dimensional space,
every mildly compact space is compact.
\end{enumerate}
\end{proof}

\begin{corollary}
For a zero-dimensional separable metric space $(X,d)$, the
following statements are equivalent:
\begin{enumerate}
\item $X$ is $\sigma$-compact; \item TWO has a winning strategy in
the game $G_{fin}(\mathcal{O}, \mathcal{O})$; \item TWO has a
winning strategy in the game $G_{fin}(\mathcal{C}_\mathcal{O},
\mathcal{C}_\mathcal{O})$; \item ONE has a winning strategy in the
mildly compact-clopen game; \item ONE has a winning strategy in
the compact-clopen game; \item ONE has a winning strategy in the
compact-open game.
\end{enumerate}
\end{corollary}

\section{$\mathcal{K}$-quasi-component-clopen game}

Now we consider a new game, namely $\mathcal{K}$-quasi-component-clopen game.

A subset $F$ of a space $X$ is called a {\it quasi-component of a
compact subset} $K$ of $X$ if $F=\bigcap \{U: U$ is clopen in $X$,
$K\subseteq U \}$.

The {\it $\mathcal{K}$-quasi-component-clopen game}
$Q_{\mathcal{K}}C(X)$ on a space $X$ is played according to the
following rules :

In each inning $n \in \omega$, ONE picks a quasi-component $A_n$
of a compact subset $K_n$ of $X$, and then TWO chooses a clopen
set $U_n \subseteq X$ with $A_n \subseteq U_n$. At the end of the
play
\begin{center}
$A_0, U_0, A_1, U_1, A_2, U_2, . . . , A_n, U_n, . . . $,
\end{center}
the winner is ONE if $X \subseteq \bigcup_{n \in \omega} U_n$, and
TWO otherwise.

We denote the collection of all quasi-components of compact subsets of a space by
$Q_\mathcal{K}$ and the collection of all finite subsets of $Q_\mathcal{K}$ by $Q_\mathcal{K}^{< \omega}$.

A strategy for ONE in the game $Q_{\mathcal{K}}C(X)$ on a space
$X$ is a function $\varphi : \tau^{< \omega}_{c} \rightarrow
Q_\mathcal{K}$.

A strategy for TWO in the game $Q_{\mathcal{K}}C(X)$ on a space
$X$ is a function $\psi : Q_\mathcal{K}^{< \omega} \rightarrow
\tau_{c}$ such that, for all $\langle A_0, A_1,..., A_n \rangle
\in Q_\mathcal{K}^{< \omega} \setminus \{\langle \rangle\}$, we
have $A_n \subseteq \psi(\langle A_0,..., A_n \rangle$) = $U_n$.

A strategy $\varphi : \tau^{< \omega}_{c} \rightarrow
Q_\mathcal{K}$ for ONE in the game $Q_{\mathcal{K}}C(X)$ on $X$ is
a winning strategy for ONE if, for every sequence $\langle U_n : n
\in \omega \rangle$ of clopen subsets of a space $X$ such that
$\forall n \in \omega$, $A_n = \varphi(\langle U_0, U_1,...,
U_{n-1} \rangle) \subseteq U_n$, we have $X \subseteq \bigcup_{n
\in \omega} U_n$. If ONE has a winning strategy in the game
$Q_{\mathcal{K}}C(X)$ on $X$, we write $ONE {\uparrow}
Q_\mathcal{K}C(X)$.

A strategy $\psi : Q_\mathcal{K}^{< \omega} \rightarrow \tau_{c}$
for TWO in the game $Q_{\mathcal{K}}C(X)$ on $X$ is a winning
strategy for TWO if, for every sequence $\langle A_n : n \in
\omega \rangle$ of quasi-components of compact subsets of a space
$X$, we have $X \subseteq \bigcup_{n \in \omega} (\psi(\langle
A_0, A_1,..., A_n \rangle) = U_n)$. If TWO has a winning strategy
in the game $Q_{\mathcal{K}}C(X)$ on $X$, we write $TWO {\uparrow}
Q_\mathcal{K}C(X)$.

\begin{proposition} \label{c2} The compact-clopen game is equivalent to the $\mathcal{K}$-quasi-component-clopen game.
\end{proposition}

\begin{proof} Let $\varphi : \tau^{< \omega}_{c} \rightarrow \mathcal{K}$ be
a winning strategy for ONE in the compact-clopen game on a space
$X$. Then the function $\psi : \tau^{< \omega}_{c} \rightarrow
Q_\mathcal{K}$ such that $\psi(\langle U_0, U_1,..., U_{n-1}
\rangle)=Q[\varphi(\langle U_0, U_1,..., U_{n-1} \rangle)]$
($Q[K]$ is a quasi-component of $K \in \mathcal{K}$) for every
sequence $\langle U_n : n \in \omega \rangle$ of clopen subsets of
a space $X$ and $n \in \omega$,  is a winning strategy for ONE in
the $\mathcal{K}$-quasi-component-clopen game. This follows from
the fact that $K_n=\varphi(\langle U_0, U_1,..., U_{n-1}
\rangle)\in Q[K_n]\subseteq U_n$.

Let $\varphi : \tau^{< \omega}_{c} \rightarrow Q_\mathcal{K}$ be a
winning strategy for ONE in the
$\mathcal{K}$-quasi-component-clopen game on a space $X$. Then the
function $\psi : \tau^{< \omega}_{c} \rightarrow \mathcal{K}$ such
that $\psi(\langle U_0, U_1,..., U_{n-1} \rangle)\in
\varphi(\langle U_0, U_1,..., U_{n-1} \rangle)$ for every sequence
$\langle U_n : n \in \omega \rangle$ of clopen subsets of a space
$X$ and $n \in \omega$,  is a winning strategy for ONE in the
compact-clopen game. This follows from the fact that if $W$ is a
clopen set of $X$ and $K \subseteq W$ then $Q[K]\subseteq W$.

Let $\psi : \mathcal{K}^{< \omega} \rightarrow \tau_{c}$ be a
winning strategy for TWO in the compact-clopen game on $X$. Then
the function $\rho : Q_\mathcal{K}^{< \omega} \rightarrow
\tau_{c}$ such that $\rho(\langle A_0, A_1,..., A_n
\rangle)=\psi(\langle K_0, K_1,..., K_n \rangle)$ for every
sequence $\langle A_n : n \in \omega \rangle$ of quasi-components
of compact subsets $K_n$ of a space $X$ and some $K_0,...,K_n$
that $A_i=Q[K_i]$ for each $i=0,...,n$, is a winning strategy for
TWO in the $\mathcal{K}$-quasi-component-clopen game.

Let $\psi : Q_\mathcal{K}^{< \omega} \rightarrow \tau_{c}$ be a
winning strategy for TWO in the
$\mathcal{K}$-quasi-component-clopen game on $X$. Then the
function $\rho : \mathcal{K}^{< \omega} \rightarrow \tau_{c}$ such
that $\rho(\langle K_0, K_1,..., K_n \rangle)=\psi(\langle A_0,
A_1,..., A_n \rangle)$ for every sequence $\langle K_n : n \in
\omega \rangle$ of points of a space $X$ where $A_i=Q[K_i]$ for
each $i=0,...,n$, is a winning strategy for TWO in the
compact-clopen game.
\end{proof}

\begin{proposition} \label{c1} Suppose that $X$ is a union of countably many quasi-components of compact sets. Then TWO
has a winning strategy in the game
$G_{fin}(\mathcal{C}_\mathcal{O}, \mathcal{C}_\mathcal{O})$.

\end{proposition}

\begin{proof}
Let $X = \bigcup_{i \in \omega} Q_{K_i}C(X)$, where $Q_{K_i}C(X)$ is a quasi-component of compact set $K_i$ for each $i$.

Let ONE starts $G_{fin}(\mathcal{C}_\mathcal{O},
\mathcal{C}_\mathcal{O})$ with his initial move $\mathcal{U}_0$, a
cover by clopen sets of $X$. Then TWO replies with a finite subset
$\mathcal{V}_0$ of $\mathcal{U}_0$ such that $K_1 \subseteq \bigcup \mathcal{V}_0$. Then $Q_{K_1}C(X) \subseteq \bigcup \mathcal{V}_0$.

If ONE plays $\mathcal{U}_1$ his next move, then TWO replies with
a finite subset $\mathcal{V}_1$ of $\mathcal{U}_1$ such that $K_2 \subseteq \bigcup \mathcal{V}_1$. Then $Q_{K_2}C(X) \subseteq \bigcup \mathcal{V}_1$.

If ONE plays $\mathcal{U}_2$ his next move, then TWO replies with
a finite subset $\mathcal{V}_2$ of $\mathcal{U}_2$ such that $K_3 \subseteq \bigcup \mathcal{V}_2$. Then $Q_{K_3}C(X) \subseteq \bigcup \mathcal{V}_2$
and so on.

If ONE plays $\mathcal{U}_n$ in $n$th inning, then TWO replies
with a finite subset $\mathcal{V}_n$ of $\mathcal{U}_n$ such that
$K_{n+1} \subseteq \bigcup \mathcal{V}_n$. Then $Q_{K_{n+1}}C(X) \subseteq \bigcup \mathcal{V}_n$.

In the same manner, we get a play of the game $G_{fin}(\mathcal{C}_\mathcal{O}, \mathcal{C}_\mathcal{O})$ :
\begin{center}
$\langle \mathcal{U}_0, \mathcal{V}_0, \mathcal{U}_1, \mathcal{V}_1,..., \mathcal{U}_n, \mathcal{V}_n,... \rangle$.
\end{center}
Since $X = \bigcup_{i \in \omega} Q_{K_i}C(X)$, $X = \bigcup_{n \in \omega} \bigcup \mathcal{V}_n$. This completes the proof.
\end{proof}

The following chain of implications always holds. Note that the
top equivalence $(*)$ follows from Telgarsky's equivalence
together with Theorem 3.2.

\begin{center}

$X$ is a union of countably many quasi-components of compact sets \\
$\Downarrow$ \\ $TWO \uparrow G_{fin}(\mathcal{C}_\mathcal{O},
\mathcal{C}_\mathcal{O})$ ${\Leftrightarrow\atop (*)}$ $ONE
\uparrow$ the compact-clopen game\\ $\Downarrow$ \\ $ONE
\not\uparrow G_{fin}(\mathcal{C}_\mathcal{O},
\mathcal{C}_\mathcal{O})$ \\ $\Updownarrow$ \\
$X$ has mildly Menger property.

\end{center}

\begin{example} Let $Z=X\times Y$ where $X$ is the one-point compactification of uncountable
discrete space $D$ and $Y$ is a connected non-$\sigma$-compact
space.
\end{example}

Then $Z$ is a quasi-component of compact set $X\times \{y\}$ for
some $y\in Y$, but $Z$ is not $\sigma$-compact and $Z$ does not
consist of countable number of quasi-components.

If possible suppose $Z$ is $\sigma$-compact. Then $Z = \cup_{i \in
\omega} X_i \times Y_i$, where $X_i$ is compact subset of $X$ and
$Y_i$ is a compact subset of $Y$ for each $i$. This means that $Y
= \cup_{i \in \omega} Y_i$ is $\sigma$-compact, a contradiction.

For each point $(x,y)$ of $X \times Y$, the quasi-component of
$(x,y)$ is $\{x\} \times Y$. Then $Z$ has uncountable number of
quasi-components. Since each $\{(x,y)\}$ is compact so it also
have uncountable number of quasi-components of compact sets.

\medskip

\begin{remark} A quasi-component of compact subset $B$ of a
zero-dimensional space $X$ is equal to $B$. It follows that, if
$X$ is countable union of quasi-components of compact subsets of
$X$ then $X$ is $\sigma$-compact.
\end{remark}

From Theorem \ref{d1}, we have the following remark.

\begin{remark} For a zero-dimensional second countable space $X$, TWO has a winning strategy in the game
$G_{fin}(\mathcal{C}_\mathcal{O}, \mathcal{C}_\mathcal{O})$ if and
only if $X$ is a $\sigma$-mildly compact space.
\end{remark}

\textbf{Determinacy and $G_{fin}(\mathcal{C}_\mathcal{O}, \mathcal{C}_\mathcal{O})$ game}

A game $G$ played between two players ONE and TWO is determined if
either ONE has a winning strategy in game $G$ or TWO has a winning
strategy in game $G$. Otherwise $G$ is undetermined.

It can be observed that the game $G_{fin}(\mathcal{C}_\mathcal{O},
\mathcal{C}_\mathcal{O})$ is determined for every $\sigma$--mildly
compact space. But in a non $\sigma$-mildly compact, mildly Menger
and a zero-dimensional metric space, none of the players ONE and
TWO have a winning strategy. Since a zero-dimensional metric
mildly Menger space is second countable,
$G_{fin}(\mathcal{C}_\mathcal{O}, \mathcal{C}_\mathcal{O})$ is
undetermined for a non $\sigma$-mildly compact, mildly Menger and
a zero-dimensional metric space. Thus every non $\sigma$-mildly
compact zero-dimensional mildly Menger metric space is
undetermined.

Recall that an uncountable set $L$ of reals is a Luzin set if for
each meager set $M$, $L \cap M$ is countable.  The Continuum
Hypothesis implies the existence of a Luzin set.

Sierpi\'{n}ski showed that Lusin sets of real numbers have the
Menger property. Since Lusin sets are not $\sigma$-compact they
are spaces where neither player has a winning strategy.

Now from Corollary 2 in \cite{H28b}, a Luzin set is an example of
a space for which the game $G_{fin}(\mathcal{C}_\mathcal{O},
\mathcal{C}_\mathcal{O})$ is undetermined.

\medskip

Then we present several questions, the answers to which will be a
natural continuation of the research within the framework of the
topic of this paper.

\medskip

{\bf Question 1.} Assume that $X,Y$ satisfy
$G_{fin}(\mathcal{C}_\mathcal{O}, \mathcal{C}_\mathcal{O})$
($S_{fin}(\mathcal{C}_\mathcal{O}, \mathcal{C}_\mathcal{O})$).
Does it follows that $X\times Y$ satisfies
$G_{fin}(\mathcal{C}_\mathcal{O}, \mathcal{C}_\mathcal{O})$
($S_{fin}(\mathcal{C}_\mathcal{O}, \mathcal{C}_\mathcal{O})$) ?

\medskip

{\bf Question 2.} Is $S_{fin}(\mathcal{C}_\mathcal{O},
\mathcal{C}_\mathcal{O})$ ($G_{fin}(\mathcal{C}_\mathcal{O},
\mathcal{C}_\mathcal{O}$) preserved by finite powers ?

\medskip

{\bf Question 3.} Are $S_{fin}(\mathcal{C}_\mathcal{O},
\mathcal{C}_\mathcal{O})$ and $G_{fin}(\mathcal{C}_\mathcal{O},
\mathcal{C}_\mathcal{O})$ hereditary for subsets representable as
a countable union of clopen sets ?

\medskip

{\bf Acknowledgements.} The authors would like to thank the
referee for careful reading and valuable comments.


\begin{thebibliography}{11}

\bibitem{H28a} L. F. Aurichi and R. R. Dias, \textit{A minicourse on topological games}, Topology Appl., \textbf{258} (2019),  305-335.




\bibitem{H34} E. Borel, \textit{Sur la classification des ensembles de mesure nulle}, Bull. Soc. Math. France, \textbf{47} (1919), 97-125.

\bibitem{H10} R. Engelking, \textit{General Topology, Revised and completed edition}, Heldermann Verlag Berlin (1989).

\bibitem{enc}
K.P.Hart, Jun-iti Nagata, J.E.Vaughan, \textit{Encyclopedia of
General Topology}, Elsevier Science, 2003, 536 p.

\bibitem{H4} W. Hurewicz, \textit{$\ddot{U}$ber eine verallgemeinerung des Borelschen Theorems}, Math. Z. \textbf{24} (1925), 401-421.

\bibitem{H5} W. Hurewicz, \textit{$\ddot{U}$ber Folgen stetiger Funktionen}, Fund. Math., \textbf{9} (1927), 193-204.


\bibitem{H2} W. Just, A.W. Miller, M. Scheepers and P.J. Szeptycki, \textit{The combinatorics of open covers (II)}, Topology Appl.,
\textbf{73} (1996), 241-266.


\bibitem{H47} K. Menger, \textit{Einige Überdeckungssätze der punktmengenlehre}, Sitzungsberischte Abt. 2a, Mathematik, Astronomie, Physik,
Meteorologie und Mechanik (Wiener
Akademie, Wien) \textbf{133} (1924), 421-444.


\bibitem{H56} A.W. Miller and D.H. Fremlin, \textit{Some properties of Hurewicz, Menger and Rothberger}, Fund. Math., \textbf{129} (1988), 17-33.



\bibitem{Ko} Lj.D. Ko\v{c}inac, \textit{On mildly Hurewicz
spaces}, Int. Math. Forum. {\bf 11}(12) (2016), 573-582.


\bibitem{H1} M. Scheepers, \textit{Combinatorics of open covers (I) : Ramsey theory},  Topology Appl.,\textbf{69} (1996), 31-62.


\bibitem{H28b} M. Scheepers, \textit{A direct proof of a theorem of Telg\'{a}rsky}, Proc. Amer. Math. Soc.,\textbf{123} (1995), 3483-3485.

\bibitem{st} R. Staum, \textit{The algebra of bounded continuous
functions into a nonarchimedean field}, Pacific J. Math., {\bf 50}
(1974), 169-185.

\bibitem{Tel}
R. Telg\'{a}rsky, \textit{On games of Topsoe}, Math. Scand. {\bf
54} (1984), 170--176.


\end{thebibliography}
\end{document}